\documentclass{amsart}

\usepackage{amssymb}

\usepackage{epsfig}  		

\newtheorem{thm}{Theorem}[section]

\theoremstyle{definition}

\theoremstyle{remark}

\numberwithin{equation}{section}

\usepackage[]{hyperref}
\hypersetup{
    pdftitle={Studies of entropy measures concerning the gaps of prime numbers},
    pdfauthor={AOT HHS},
    pdfsubject={Multiplicative Number Theory},
    pdfkeywords={Entropy, Prime Numbers},
    bookmarksnumbered=true,     
    bookmarksopen=true,         
    bookmarksopenlevel=1,       
    colorlinks=true,            
    pdfstartview=Fit,           
    pdfpagemode=UseOutlines,    
    pdfpagelayout=TwoPageRight
}

\begin{document}


\title{Studies of entropy measures concerning the gaps of prime numbers}


\author{Arturo Ortiz Tapia}
\address{Mexican Petroleum Institute}
\email{aortizt@imp.mx}


\author{Hans Henrik St\o{}leum}



\begin{abstract}
The Shannon entropy is used as a basis for applying different lemmas and conjectures concerning the set of gaps between prime numbers $G_p$, thus estimating several measures of it. The same procedures are applied to artificially created number sets, to compare the size of their entropy against $G_p$.
\end{abstract}


 \maketitle


\tableofcontents

\section{Introduction}
Shannon's (information) entropy is based on probabilities. If a distribution of probabilities is known, it can be estimated using the formula for discrete distributions or for continuous distributions \cite{fouque2011close}.

Why entropy bounds may not be good enough? Previous works on the generation of prime numbers, such as \cite{brandt1992generation,joye2006fast,maurer1995fast}, provide a proof (based on rather strong assumptions) that the output distribution of their algorithm has an entropy not much smaller than the entropy of the uniform distribution. This is a reasonable measure of the inability of an adversary to guess which particular prime was output by the algorithm, but it doesn't rule out the possibility of gaining some information
about the generated primes. In particular, it doesn't rule out the existence of an efficient distinguisher between the output distribution and the uniform one. For example \cite{joye2006fast}, let $H_{max}...$ nbit prime. $n \geq 256 $
\begin{equation}
H_{max} - H < \frac{1-\gamma}{\log 2} = 0.609949
\end{equation}
$\gamma$ is the Euler-Mascheroni constant.
The entropy loss with respect to a perfectly uniform generation is less than 0.61 bit for any prime bit length.\\

Therefore the central part of calculating Shannon's entropy rest upon how does one obtain the probabilities.
For example \cite{minculeta2011entropy} uses 
\begin{equation}
H(n) = \log \Omega (n) - \frac{1}{\Omega (n)} \sum_{i=1}^{r} a_i \log a_i
\end{equation}
where $\Omega (n)$ is the sum of the total number of prime factors of the natural number $n$ and $a_i$ is the multiplicity of each one of those prime factors. In particular for $n=2 \cdot 3^{2}\cdot 5^{3}$ the entropy is estimated as
\begin{equation}\label{EcShannon1}
H(n) = \log 6 - \frac{1}{6}(2\log 2 + 3 \log 3) \approx 1.011
\end{equation}

In this work it usage will be made of several measures like Eq.\ref{EcShannon1} to prove that the entropy of a given amount of gaps between prime numbers is less than the entropy of a uniform distribution of a similar quantity of natural or real numbers. 

%

\section{Entropy of Real Numbers}

Shannon's entropy is defined as \cite{shannon1948mathematical}

\begin{equation}
H(x)=-\sum_{i=1}^{n} P(x_i)\log_b P(x_i)
\end{equation}
where $P(x_i)$ is the probability mass funcion of the discrete random variable $X$ with possible values $\{x_1,\cdots,x_n\}$ and $\log_b$ is the logarithm in base $b$ used.

For the continuous case we have \cite{park2009maximum}

\begin{equation}\label{EqShannonContinuous}
h(X)=-\int_{\mathbb{X}}f(x)\log f(x)\mathrm{d}x
\end{equation}

where $f$ is a probability density function whose support is a set $\mathbb{X}$

For the purposes of this work, suppose it is desired to measure the entropy of a set with a uniform distribution with a given amount of real numbers (a set of random reals $R_{real}$), and in order to make it congruent to any measure of entropy with prime numbers, the support set $\mathbb{X} = [2,N_p]$, where $N_p$ is a given number of primes up to a given bound. Using Eq.\ref{EqShannonContinuous}, the entropy of real numbers can be defined as 

\begin{equation}
H(R_{real})=-\int_2^{Np}\frac{1}{Np-2}\log \left( \frac{1}{Np-2}\right)dR_{real} =\log(Np-2)
\end{equation}

By the Prime Number Theorem \cite{hardy1916contributions}, the amount of prime numbers $N_p$ up to a given bound $x$ is
\begin{equation}\label{approxnumprimes}
N_p \sim \frac{x}{\ln x}\sim \pi (x)
\end{equation}
$\pi (x)$ the number of primes less than or equal to x, for any Real number x. Assume that for a given, concrete, measure, we have $x = x_{max}$, and using Eq.\ref{EqShannonContinuous}
\begin{equation}\label{EqEntropyReals}
H(Real) = -\int_2^{\frac{x_{max}}{\ln x_{max}}} \frac{1}{\frac{x_{max}}{\ln x_{max}} -2}\ln \left(\frac{1}{\frac{x_{max}}{\ln x_{max}} -2} \right)dx =\ln \left(\frac{x_{max}}{\ln x_{max}} -2\right)
\end{equation}
for a chosen maximal x.

Goldston, Pintz, and Yildirim \cite{goldston2009primes} under appropiate unproved conjectures, showed that there are infinitely many primes $P_n,\, P_{n+1}$ such that: $P_{n+1}-P_n < 16$. Therefore, applying Eq.(\ref{EqEntropyReals}) it can be shown thar 
\begin{equation}
H(Real)=\ln \left( \frac{x_{max}}{\ln x_{max}} -2\right) >16 
\end{equation}

infinitely often, approximately for $x_{max} > 67.3611$

\section{Usage of Cram\'er's conjecture}
Let $G(x)$ denotes the largest gap between consecutive primes below a given bound $x$ \cite{wolf2011some}:

%
%
%

\begin{equation}\label{Gx1}
G(x) \sim \log (x) \left(\log (x) - 2 \log \log(x)  + c \right);\, c=0.2778769 
\end{equation}

and $x = x_{max}$.
For large x, let us take Cram\'er's conjecture\cite{cramer1936order} 
\begin{equation}\label{EqCramersimlogsquare}
G(x) \sim \log^2 (x)  
\end{equation}      
so
\begin{eqnarray}\nonumber
H(G(x))=-\int_1^{x_{max}} \log^2 (x_{max})\log\left( \frac{1}{\log^2 (x_{max}-1) }\right) dx  &=& -\log^2 (x_{max}) \log \left( \frac{1}{\log^2 (x_{max}-1)}\right)(x_{max}-1) \\ \nonumber
 &=& \log^{-2} (x_{max}) \log \left( \frac{1}{\log^2 (x_{max}-1)}\right)(x_{max}-1) \\ \nonumber
 &=& \log \left( \frac{1}{\log^2 (x_{max}-1)}\right)(x_{max}-1) / \log^{2}
\end{eqnarray}

For comparison, we generate randomly uniform distributions of gaps from 2 to $G(x_{max})$
so 
\begin{eqnarray}\nonumber
H(Random\, gaps) &=& -\int_2^{G(x_{max})} \frac{1}{G(x_{max})-2}\log\left(\frac{1}{G(x_{max})-2}\right) dx\\\nonumber
&=& -\frac{1}{G(x_{max})-2} \log\left(\frac{1}{G(x_{max})-2}\right)(G(x_{max})-2)\\\nonumber
&=& -\log\left(\frac{1}{G(x_{max})-2}\right)\\\nonumber
&=&\log\left(G(x_{max})-2\right)
\end{eqnarray}

Thus resulting the inequality to be questioned
\begin{equation}
\log\left(\frac{x_{max}}{\log x_{max}} - 2\right) > \log\left(G(x_{max})-2\right) \, \forall x > x_{min} ?
\end{equation}
using Eq. \ref{Gx1} we obtain
\begin{equation}
\log\left(\frac{x_{max}}{\log x_{max}} - 2\right) > \log\left(\log (x_{max}) \left(\log (x_{max}) - 2 \log \log (x_{max}) + c \right) -2\right) 
\end{equation}
exponentiating both sides and scratching out the outermost "2"s we get

\begin{equation}
\frac{x_{max}}{\log x_{max}} > \log (x_{max}) \left(\log (x_{max}) - 2 \log \log (x_{max}) + c \right)
\end{equation}              
\begin{equation}
x_{max} > \log^2 (x_{max}) \left(\log (x_{max}) - 2 \log \log (x_{max}) + c \right)
\end{equation}
solving for $x_{max}$
\begin{equation}
x_{max} > 9.17162
\end{equation}
Meaning that for $x_{max} > 9.17162$, the entropy of the largest gap from those generated with a uniform distribution, is greater than the entropy of the largest gap between primes, below a given bound $x$, according to Cram\'er's conjecture.

Moreover, for a uniformly distribution of gaps from 2 to max gaps =  $G(x_{max})$, we will have always
\begin{equation}
\log\left(G(x_{max})-2\right)
\end{equation}               
whatever $G(x_{max})$ is; taking now $G(x_{max}) \sim \log^2 (x_{max})$ (Cramer's conjecture) for $x >>1$ (large x) we have the inequality
\begin{equation}
\log \left( \frac{x_{max}}{\log x_{max}} -2\right) > \log \left(  \log^2 (x_{max}) -2\right)
\end{equation}

and solving for $x_{max}$, we obtain
\begin{equation}
x_{max} > 93.3545
\end{equation}   

Now, let 
\begin{equation}
G(x_{max}) = \log (x_{max})\left( \log (x_{max}) + \log \log \log (x_{max})\right)
\end{equation}

\section{Heath-Brown conjecture}           
Using a conjecture in \cite{heath1982gaps} (assuming the validity of the Riemann Hypothesis), we have that $\log\left( G(x_{max})-2\right)$ again, we get
\begin{equation}
\log \left( \frac{x_{max}}{\log x_{max}} -2\right) > \log\left(\log (x_{max})\left( \log (x_{max}) + \log \log \log (x_{max})\right)-2\right)
\end{equation}

which is always true for $x >120.027$ (or for $5.69781 < x < 8.43901$).\\

\section{Granville's formula}
A. Granville's formula \cite{granville1995harald}, which claims to be true for infinitely many pairs of primes $P_n, \, P_{n+1}$ for which 

\begin{equation}
P_n, \, P_{n+1} = G(P_n) > 2 e^{-\gamma} \log^2 (P_n) = 1.12292...\log^2 (P_n)
\end{equation}
where $\gamma = 0.577216...$ is the Euler-Mascheroni constant. 
The problem with Granville's, and other similar results, is that the formula 
\begin{equation}\label{logOverlogminus2}
\log(x_{max}/\log x_{max}) -2
\end{equation}
 cannot be used for comparison of the entropy with the reals, because Eq.(\ref{logOverlogminus2}) is not in terms of $P_n$. Said in another way, Eq.(\ref{logOverlogminus2}) is not a formula saying something about how to find directly $P_n$, i.e., how large $x_{max}$ must be, so that $P_n$ is found. However there might a way to interpret this $P_n$ using a generalization of Cram\'er's conjecture (see next section).\\

\section{Using a Generalization of Cram\'er's conjecture}

%
 Assuming the Riemann Hypothesis, Cram\'er proved that \cite{sinha2010new}
   
\begin{equation}
P_{n+1}-P_{n} = \mathcal{O} \left( P_n^{0.5} \log P_n\right)
\end{equation} 
Compare with the other formula given above ($G(x_{max}) \sim \log^2 (x_{max})$ -Eq.\ref{EqCramersimlogsquare}-); this would lead to rethink a real number, the size of $P_n$, i.e., $x_{max}=P_n$; what is the probability of $P_n$ ?\\
As a consequence of the Prime Number Theorem (PNT), one gets an asymptotic expression for the $n^{th}$ prime number, denoted  by  $P_n$ \cite{ribenboim2004little}:
\begin{equation}\label{PNTconsequence}
P_n \sim n \log n
\end{equation}
where the right-hand-side of Eq.(\ref{PNTconsequence}) is the "size" of $P_n$ and $n$ is the index, or the ordered position in the sequence of primes, for $P_n$. Therefore we can take this as $x_{max}$. Go back to the entropy of Reals, taking $n \log n$ as the largest  number of primes, but now in terms of the index $n$.

\begin{eqnarray}
H(Real) &=& - \int_2^{\# \,of \,primes} \frac{1}{\#P -2} \log \left( \frac{1}{\#P -2}  \right) dReal \\\nonumber
        &=& -\frac{1}{\#P -2}\log \left( \frac{1}{\#P -2}  \right)(\#P -2) \\\nonumber
        &=&\log \left(\#P -2\right) \\\nonumber
\end{eqnarray}
we recall that PNT also says that the number of primes is approximately

\begin{equation}
\# primes \sim  \frac{x}{\log x}
\end{equation}
but now we are taking $x$ as $x_{max}$ and in turn this maximal number making it equal to the size of the largest prime, that is $x_{max}=n\log n$, so first

\begin{equation}
\log\left(\frac{x_{max}}{\log x_{max}} -2\right)
\end{equation}
and substituting the value of $x_{max}$ and by Eq.(\ref{PNTconsequence})
\begin{equation}
\log\left(\frac{n\log n}{\log (n\log n)}- 2\right) = \log\left(\frac{P_n}{\log P_n}-2\right) 
\end{equation}

and now this new formula for the size of entropy in terms of the index $n$ can be used for those formulas of the gaps that are in terms of $P_n$, being now $P_n \sim n \log n$.
We go back to Granville's formula given by Wolf, in terms of the Euler-Mascheroni constant.
\begin{equation}
G(P_n) > 2 e^{-\gamma} \log^2 (P_n) = 1.12292...\log^2(P_n)
\end{equation}
We use  $\log(G(P_n)-2)$ so we are saying that either
\begin{equation}
\log\left(\frac{P_n}{\log P_n}-2\right) > \log\left(2 e^{-\gamma} \log^2 (P_n)-2\right)
\end{equation}
or that (by \ref{PNTconsequence})
\begin{equation}
\log\left(\frac{n\log n}{\log (n\log n)}-2\right) > \log\left(2 e^{-\gamma} \log^2 (n\log n) -2\right)
\end{equation}
Solving numerically for $P_n\; ; P_n \in \mathbb{R}$ gives $P_n > 128.703$ and the next prime to this real number is 131.
The best current unconditional result for $P_{n+1} - P_n$ is $\mathcal{O}\left( P_n^{0.535}\right)$ due to R. Baker and G. Harman, 1996, so:
 \begin{equation}
\log\left(\frac{P_n}{\log P_n}-2\right) > \log\left(P_n^{0.535}-2\right)
\end{equation}
which renders $P_n > 3.6532$ or $P_n >5$
or, in terms of the index $n$
\begin{equation}
\log\left(\frac{n\log n}{\log (n\log n)}-2\right) > \log\left(\left(n\log n\right)^{0.535} -2\right)
\end{equation}

\section{Cram\'er's conjecture}
\begin{equation}
P_{n+1} - P_n =\mathcal{O}\left( P_n^{0.5} \log P_n\right)
\end{equation}
assuming the Riemann Hypothesis.
so the entropy inequality will be written as
\begin{equation}
\log\left(\frac{P_n}{\log P_n}-2\right) > \log\left(P_n^{0.5}\log P_n-2\right)
\end{equation}

so $P_n > 5503.66$ or $P_n > 5507$.
\section{Stronger Form  of Firoozbakt's Conjecture}
Sinha \cite{sinha2010new} deduces  a stronger form  of Firoozbakt's conjecture, from which it is deduced that 
\begin{equation}
P_{n+1} -P_n < \log^2 P_n - \log P_{n+1}
\end{equation}
the right hand side of the inequality will be taken as  $G(P_n)$.\\

Now we are saying that \cite{ribenboim2004little}
\begin{equation}\label{SinhaRibeboim}
\log\left(\frac{P_n}{\log P_n}-2\right) >\log\left(\log^2 P_n -2 \log P_{n+1} -2 \right)
\end{equation}
which is true for $P_n\geq 17\wedge P_{n+1}\geq 19$, or
\begin{equation}
\log\left(\frac{n\log n}{\log (n\log n)}-2\right) >\log\left(\log^2 (n\log n) -2 \log ((n+1)\log (n+1)) -2 \right)
\end{equation}
which is true for $n\geq 9$.\\

\section{Upper bound of Jaroma 2005}
Instead of Eq.(\ref{SinhaRibeboim}) the following can also be used \cite{jaroma2005upper}
\begin{equation}
P_n < (1.2)^k < P < (1.2)^{k+1}
\end{equation}
Therefore $P_{k+1} < (1.2)^{k+1}$, which we substitute in Eq.(\ref{SinhaRibeboim})
\begin{equation}\label{Jaroma1}
\log\left(\frac{(1.2)^n}{\log (1.2)^n}-2\right) >\log\left(\log^2 (1.2)^n -2 \log (1.2)^{n+1} -2 \right)
\end{equation}
Solving for $n$, we obtain $n \in \mathbb{+Z} \wedge n \geq 16$.\\

\section{Estimations based on Tschebychef function, Robin 1983}

We have been using 
\begin{equation}\label{basisIneq}
\log\left(\frac{P_n}{\log P_n}-2\right) > \log \left( G(P_n) \right)
\end{equation}
suppose now the following result \cite{robin1983estimation}
\begin{equation}\label{robinEq1983}
P_n \leq n\log n + n \left( \log \log n -0.9385 \right)
\end{equation}

for $n \geq 7022$
substituing the right hand side of Eq.(\ref{robinEq1983}) on the left hand side of Eq.(\ref{SinhaRibeboim})
\begin{equation}\label{lhsSinhaRibenboim}
\log\left(\frac{n\log n + n \left( \log \log n -0.9385 \right)}{\log n\log n + n \left( \log \log n -0.9385 \right)}-2\right) > \log\left(\log^2 P_n -2 \log P_{n+1} -2 \right)
\end{equation}

Even if we don't have a result like Eq.(\ref{robinEq1983}) for $P_{n+1}$, let us argue like Jaroma \cite{jaroma2005upper} and suppose that
\begin{equation}\label{robinJaromaNextP}
P_{n+1} \leq (n+1)\log (n+1) + (n+1) \left( \log \log (n+1) -0.9385 \right)
\end{equation}

substituting Eq.(\ref{robinEq1983}) and Eq.(\ref{robinJaromaNextP})  in Eq.(\ref{basisIneq}) we obtain

\begin{eqnarray}
\log\left(\frac{n\log n + n \left( \log \log n -0.9385 \right)}{\log n\log n + n \left( \log \log n -0.9385 \right)}-2\right) &>& \\\nonumber
 \log\left(\log \left(n\log n + n \left( \log \log n -0.9385 \right)\right)^2 - 2 \log (n+1)\log (n+1) + (n+1) \left( \log \log (n+1) -0.9385 \right) -2 \right)\\\nonumber
\end{eqnarray}

But this is true only for $16\leq n \leq 32$.\\

\section{Upper and lower bounds for gaps between primes}
\subsection{kontorovich-Zhang 2013}
Alex Kontorovich \cite{kontorovich2014levels} explains the basis of Zhang's theorem: 
In April 2013 Zhang \cite{zhang2014bounded} proved that the Bounded Gaps Conjecture is true.
In particular,
\begin{equation}
\liminf_{n \rightarrow \infty} (P_{n+1}-P_n) < 7\times 10^7
\end{equation}

The average gap $P_{n+1} - P_n$ is of size about $\log P_n$ (by PNT):
\begin{equation}
\forall x,\, x >> 1, \, \exists_{(P_{n+1},\, P_n)} \left((P_{n+1}-P_n) < 7\times 10^7\right)
\end{equation}
in the range $[x,\, 2x]$ \cite{zhang2014bounded}. This will be taken as an upper limit for calculating the entropy of gaps between prime numbers.\\

\subsection{Kontoyiannis 2008}
Ioannis Kontoyiannis \cite{kontoyiannis2008counting}, cites a paper of Chebyshev, 1852 where
\begin{equation}
C(n)\triangleq \sum_{P\leq n} \frac{\log P}{P}\sim \log n
\end{equation}
where the sum is over all primes not exceeding $n$, and furthermore proves that
\begin{equation}\label{EqKontoyiannis}
\liminf_{n \rightarrow \infty}\frac{C(n)}{\log n} =1,
\end{equation}
using as one of the arguments that the entropy contained in the the number of primes up to certain value, is inversely proportional to the number of primes within that range:
\begin{equation}
H(N)=\log\frac{1}{P(N)}
\end{equation}

\subsection{Perepelyuk 2013}

In \cite{perepelyuk2013distribution} it is mentioned that
\begin{equation}\label{Eqpereplyuk}
\liminf_{n \rightarrow \infty} \left(\frac{P_{n+1} - P_n}{\log P_n}\right) = 0 
\end{equation}
meaning that $\log P_n$ strictly bounds from below the size of the gaps. Eqns. \ref{EqKontoyiannis} and \ref{Eqpereplyuk} tell us that \emph{on average} one can expect that the lower bound of gaps between prime numbers is as low as possible. Since prime numbers are odd numbers, the smallest possible gap is 2 and using $7 \times 10^7$ as an upper bound, we now need to estimate a smooth envelope for prime gaps.\\

\subsection{Smooth envelope}

Using Merten's theorem \cite{Mertens1874} and PNT , it can be shown that
\begin{equation}
f(k) = \mathcal{O}(\log\log k)
\end{equation}
where 
\begin{equation}\label{prodf}
f(k) = \prod_{P|k,\, P>2}\frac{P-1}{P-2}
\end{equation}
($k$ even) and compare for different $k$. So $f(6) = 2>1 =f(8)$; so gaps of length 6 are asymptotically twice as common as gaps of length 8.

So, using Eq.(\ref{prodf}), for 6, we only have 3 as prime:
\begin{equation}
 \prod_{P|k,\, P>2}\frac{P-1}{P-2}=\frac{3-1}{3-2}=2
\end{equation}
8 is divided by 8, 4, 2 but 2 is the only prime, and cannot be used. We should have zero, but by definition the function resorts to 1. We have, then that 
\begin{equation}
f(k) = \prod_{P|k,\, P>2}\frac{P-1}{P-2}= \mathcal{O}(\log\log k)
\end{equation}
The right-hand-side of this last expression is our smooth "envelope", with $k$ meaning gaps of any length. By PNT and Merten's theorem, $G(x) > \log (x)$, and using the bounds $2$ and $7 \times 10^7$,  an estimate of the entropy of gaps between prime numbers is:
\begin{equation}
H(G(P))=-\int_2^{7 \times 10^7}\left(\frac{1}{\log\left(\log (k)\right)}\right)\log \left(\frac{1}{\log\left(\log (k)\right)}\right)dk
\end{equation}
(assuming smoothness)
Solving we obtain $2.57231 \times 10^7$.\\

Again, by PNT, $G(x) > \log (x)$. The $x$ in this last equation is the natural number, which has to be so large, we can accomodate the size of the maximum possible entropy of gaps.
We can go back to the formula for entropy of reals.
\begin{equation}
H(Real) = \log \left( \frac{x_{max}}{\log x_{max}}-2\right)
\end{equation}
for a chosen maximal $x$.
We know that $maxgap = 7 \times 10^7 = G(x) > \log (x)$; solving for $x$,
we obtain the exact number $x=e^{7 \times 10^7}$, meaning that, provided $x \geq \exp [7 \times 10^7]$, the entropy of the gaps will remain forever lower than any given real number.
For a uniform distribution of gaps, the entropy is
\begin{equation}
\log \left( G(x_{max}) -2\right)
\end{equation}
we know from \cite{zhang2014bounded} that $G(x_{max}) = 3 \times 10^7$, and that $G(x) > \log (x)$, so:
\begin{equation}
\log\left( (x_{max}) -2 \right)
\end{equation}
is the \textit{minimum} entropy attainable by a random gap generator, which doesn't "know" about the maximal gap. The entropy is therefore:
\begin{equation}
x =e^{2+e^{3\times 10^7}}
\end{equation}
and from there onwards the entropy of random gaps must be larger, and for $H(Real)$
\begin{equation}
\log\left(\frac{\log(x)}{\log(\log x)}-2\right) \geq 3\times 10^7
\end{equation}

For $x >>1$:
\begin{equation}
\log\left(\frac{\log(x)}{\log(\log x)}-2\right) \approx \log\left(\frac{\log(x)}{\log(\log x)}\right)
\end{equation}
which we require to be greater than $3\times 10^7$, so:
\begin{equation}
\frac{\log(x)}{\log(\log x)} \geq \exp [3\times 10^7]
\end{equation}
and we know that
\begin{equation}
\frac{\log(x)}{\log(\log x)} < \exp (\exp (x))
\end{equation}
therefore, it is sufficient that
\begin{equation}
x \geq \exp \left(\exp \left( \exp [3\times 10^7] \right) \right)
\end{equation}
for $H(Real)$ to be always greater than the entropy of the gaps between prime numbers.

\section{$H(P_{n+1}-P(n)=\min\{H(x_{k+1}-x_k)\},\,k\in \mathbb{R}$}

\begin{thm} The entropy of the gaps between prime numbers is smaller than any similar distribution made with random gaps, and of real numbers.
\end{thm}


\begin{proof}
Using the results of all sections, it is shown that at least above a certain measure, it can be certain that the entropy of gaps between prime numbers will remain smaller than any random distribution of gaps of similar size, or of real numbers of comparable size.
\end{proof}

 \bibliographystyle{siam} 
 \bibliography{EntropyGaps.bib}

\begin{thebibliography}{10}

\bibitem{brandt1992generation}
{\sc J.~Brandt and I.~Damg{\aa}rd}, {\em On generation of probable primes by
  incremental search}, in Advances in Cryptology—Crypto’92, Springer, 1992,
  pp.~358--370.

\bibitem{cramer1936order}
{\sc H.~Cram{\'e}r}, {\em On the order of magnitude of the difference between
  consecutive prime numbers}, Acta Arithmetica, 2 (1936), pp.~23--46.

\bibitem{fouque2011close}
{\sc P.-A. Fouque and M.~Tibouchi}, {\em Close to uniform prime number
  generation with fewer random bits.}, IACR Cryptology ePrint Archive, 2011
  (2011), p.~481.

\bibitem{goldston2009primes}
{\sc D.~A. Goldston, J.~Pintz, and C.~Y. Yildirim}, {\em Primes in tuples i},
  Annals of Mathematics,  (2009), pp.~819--862.

\bibitem{granville1995harald}
{\sc A.~Granville}, {\em Harald cram{\'e}r and the distribution of prime
  numbers}, Scandinavian Actuarial Journal, 1995 (1995), pp.~12--28.

\bibitem{hardy1916contributions}
{\sc G.~H. Hardy and J.~E. Littlewood}, {\em Contributions to the theory of the
  riemann zeta-function and the theory of the distribution of primes}, Acta
  Mathematica, 41 (1916), pp.~119--196.

\bibitem{heath1982gaps}
{\sc D.~Heath-Brown}, {\em Gaps between primes, and the pair correlation of
  zeros of the zeta-function}, Acta Arithmetica, 41 (1982), pp.~85--99.

\bibitem{jaroma2005upper}
{\sc J.~H. Jaroma}, {\em An upper bound on the nth prime}, The College
  Mathematics Journal, 36 (2005), p.~158.

\bibitem{joye2006fast}
{\sc M.~Joye and P.~Paillier}, {\em Fast generation of prime numbers on
  portable devices: An update}, in Cryptographic Hardware and Embedded
  Systems-CHES 2006, Springer, 2006, pp.~160--173.

\bibitem{kontorovich2014levels}
{\sc A.~Kontorovich}, {\em Levels of distribution and the affine sieve}, arXiv
  preprint arXiv:1406.1375,  (2014).

\bibitem{kontoyiannis2008counting}
{\sc I.~Kontoyiannis}, {\em Counting the primes using entropy}, IEEE
  Information Theory Society Newsletter,  (2008), pp.~6--9.

\bibitem{maurer1995fast}
{\sc U.~M. Maurer}, {\em Fast generation of prime numbers and secure public-key
  cryptographic parameters}, Journal of Cryptology, 8 (1995), pp.~123--155.

\bibitem{Mertens1874}
{\sc F.~Mertens}, {\em {Ein Beitrag zur analytischen Zahlentheorie}}, Journ.
  Reine Ang. Math., 78 (1874), pp.~46--62.

\bibitem{minculeta2011entropy}
{\sc N.~Minculeta and C.~Pozna}, {\em The entropy of a natural number}, Acta
  Technica Jaurinensis, 4 (2011), pp.~425--431.

\bibitem{park2009maximum}
{\sc S.~Y. Park and A.~K. Bera}, {\em Maximum entropy autoregressive
  conditional heteroskedasticity model}, Journal of Econometrics, 150 (2009),
  pp.~219--230.

\bibitem{perepelyuk2013distribution}
{\sc A.~Perepelyuk}, {\em On distribution of prime numbers}, PhD thesis,
  Central European University, 2013.

\bibitem{ribenboim2004little}
{\sc P.~Ribenboim}, {\em The little book of bigger primes}, Springer Science \&
  Business Media, 2004.

\bibitem{robin1983estimation}
{\sc G.~Robin}, {\em Estimation de la fonction de tchebychef $\theta$ sur le
  k-i{\`e}me nombre premier et grandes valeurs de la fonction $\omega$ (n)
  nombre de diviseurs premiers de n}, Acta Arithmetica, 42 (1983),
  pp.~367--389.

\bibitem{shannon1948mathematical}
{\sc C.~Shannon}, {\em A mathematical theory of communication, bell system
  technical journal 27: 379-423 and 623--656}, Mathematical Reviews
  (MathSciNet): MR10, 133e,  (1948).

\bibitem{sinha2010new}
{\sc N.~K. Sinha}, {\em On a new property of primes that leads to a
  generalization of cramer's conjecture}, arXiv preprint arXiv:1010.1399,
  (2010).

\bibitem{wolf2011some}
{\sc M.~Wolf}, {\em Some heuristics on the gaps between consecutive primes},
  arXiv preprint arXiv:1102.0481,  (2011).

\bibitem{zhang2014bounded}
{\sc Y.~Zhang}, {\em Bounded gaps between primes}, Annals of Mathematics, 179
  (2014), pp.~1121--1174.

\end{thebibliography}

\end{document}